\documentclass[11pt]{amsart}
\usepackage[margin=3.17cm]{geometry}

\usepackage{amssymb}
\usepackage{amsmath}
\usepackage{amsthm}
\usepackage{enumerate}
\usepackage{enumitem}

\theoremstyle{plain}
\newtheorem{theorem}{Theorem}[section]
\newtheorem{lemma}[theorem]{Lemma}
\newtheorem{proposition}[theorem]{Proposition}
\newtheorem{corollary}[theorem]{Corollary}
\newtheorem{conj}[theorem]{Conjecture}
\theoremstyle{definition}
\newtheorem{definition}[theorem]{Definition}
\newtheorem{remark}[theorem]{Remark}
\newtheorem{construction}[theorem]{Construction}
\newtheorem{notation}[theorem]{Notation}
\newtheorem{example}[theorem]{Example}

\title{Tame torsion and the tame inverse Galois problem}
\author{Matthew Bisatt and Tim Dokchitser}
\address{Fry Building, University of Bristol, Woodland Road, Bristol, BS8 1UG, UK}
\email{matthew.bisatt@bristol.ac.uk}
\email{tim.dokchitser@bristol.ac.uk}

\newenvironment{psmallmatrix}{\bigl(\begin{smallmatrix}}{\end{smallmatrix}\bigr)}

\date{\today}

\begin{document}
\global\long\def\OK{\mathcal{O}_K}
\global\long\def\CK{\mathbb{C}_K}
\global\long\def\zp{\mathbb{Z}_p}
\global\long\def\Z{\mathbb{Z}}
\global\long\def\pp{\mathbb{P}}
\global\long\def\Q{\mathbb{Q}}
\global\long\def\fp{\mathbb{F}_p}
\global\long\def\Bg{B_{\gamma}}
\global\long\def\zg{z_{\gamma}}
\global\long\def\diam{\operatorname{diam}}
\global\long\def\id{\operatorname{id}}
\global\long\def\PGL{\operatorname{PGL}}
\global\long\def\End{\operatorname{End}}
\global\long\def\Gal{\operatorname{Gal}}
\global\long\def\GSp{\operatorname{GSp}}
\global\long\def\Sp{\operatorname{Sp}}
\global\long\def\Frob{\operatorname{Frob}}

\begin{abstract}
	Fix a positive integer $g$ and a squarefree integer $m$. We prove the existence of a genus $g$ curve $C/\Q$ such that the mod $m$ representation of its Jacobian is tame.
	The method is to analyse the period matrices of hyperelliptic Mumford curves, which could be of independent
	interest. As an application, we study the tame version of the inverse Galois problem for symplectic matrix groups over finite fields.
\end{abstract}

\llap{.\hskip 10cm}\vskip -3mm
\maketitle

\section{Introduction}

We say that a number field $F$ is \emph{tame} if $F/\Q$ is tamely ramified
at every finite prime of $F$,~and \emph{wild} otherwise.
The first result of this paper concerns the problem of finding, for fixed $g$ and $m$, 
a (non-singular projective) curve $C$ of genus~$g$ whose Jacobian $J_C$ has tame 
$m$-torsion field $\Q(J_C[m])$.

\begin{theorem}[=\ref{tamecor}]
\label{tamethm}
	For every $g\geqslant 1$ and squarefree $m\geqslant 1$, there is a curve $C/\Q$ 
	of genus $g$ such that $\Q(J_C[m])$ is tame.
\end{theorem}

If $m$ is not squarefree, then $\Q(J_C[m])$ is wild, as it contains
$\Q(\zeta_m)$ (by the Weil pairing), which is wild above primes $p$ for which $p^2|m$. In that sense the result is the best possible.

Our strategy will be to reduce to the case where $m=p$ is prime and show that it suffices to construct a curve whose $p$-torsion of the Jacobian is tamely ramified at $p$, which we then do with Mumford curves. To illustrate our Mumford curve approach to this problem, we explain the idea in the elliptic curve setting in the following example.

\begin{example}
\label{ECexample}
Let $E/\Q_p$ be an elliptic curve with split multiplicative reduction. Then $E$ is isomorphic to a Tate curve and 
$E(\overline{\Q}_p) \cong {\overline{\Q}_p^{\times}}/{q^{\Z}}$ as $\Gal(\overline{\Q}_p/\Q_p)$-modules, for some $q \in p\Z_p$. Moreover any such $q$ gives rise to a Tate curve. 
In particular, $\Q_p(E[p])=\Q_p(\zeta_p,q^{1/p})$, and so, whenever $q$ is a $p$-th power (say $q=p^p$),
the extension $\Q_p(E[p])/\Q_p$ is tamely ramified.
\end{example}

\smallskip

For our second result, recall that the classical inverse Galois problem asks, given a finite group $G$, 
if there is a Galois extension $F/\Q$ such that $\Gal(F/\Q) \cong G$? This is open in general, but known for certain classes of groups including soluble groups and $G=S_n, A_n, \GSp_{2g}(\fp)$.
Birch \cite[p.35]{Bir94} further asked whether $F$ can also be taken to be tame? This is known as the tame inverse Galois problem.

We address this problem for $G=\GSp_{2g}(\fp)$, $p$ odd.
It is known when $g\!=\!1$ (all $p$) and~$g\!=\!2$ ($p \geqslant 5$) thanks to the work of 
Arias-de-Reyna--Vila \cite[Theorem 1.2]{AV09}, \cite[Theorem 5.3]{AV11}.

\begin{theorem}[=\ref{IGPcor}]
\label{tamethm2}
	 Fix a positive integer $g$ and an odd prime $p$, such that there is a Goldbach triple for $2g+2$ not containing $p$. 
	 There is a curve $C/\Q$ of genus $g$ such that $\Q(J_C[p])$ is tame, and 
	 $\Gal(\Q(J_C[p])/\Q) \cong \GSp_{2g}(\fp)$.
\end{theorem}

\noindent
See Conjecture \ref{Gold} for the definition of a Goldbach triple. This is a (slightly) strengthened version 
of the Goldbach conjecture that predicts that such triples always exist. 
On the numerical side, we show that, consequently, $\GSp_{2g}(\fp)$ is tamely realisable as a Galois group over $\Q$ 
for $g \leqslant 10^7$ and all $p>2$; see Lemma \ref{smallIGP} and the discussion afterwards.

\medskip
\noindent 
\textbf{Layout.} 
In \S 2--5 we address the tame torsion question (Theorem \ref{tamethm}). 
Specifically, in \S 2 we reduce the tame torsion question to odd primes $p$ and show that it suffices to construct a curve whose $p$-torsion is tamely ramified at $p$. We then focus on Mumford curves, which are the rigid space generalisation of the Tate curve we used in Example \ref{ECexample}. In \S 3, we review hyperelliptic Mumford curves and gather some basic results. We then compute an approximation to the period matrix in \S 4 and construct a suitable Mumford curve in \S 5.
In \S 6 we give the application to the inverse Galois problem (Theorem \ref{tamethm2}).

\begin{remark}
Ensuring that the mod $p$ representation is tamely ramified at $p$ may also be done via imposing restrictions on the endomorphism algebra instead; for details, see \cite{Bis20}. Moreover the author realises $\GSp_{2g}(\fp)$ as a Galois group over $\Q$ for all $g$ and odd primes $p$ via a non-constructive density argument \cite[Theorem 1.3]{Bis20}.
\end{remark}

\begin{notation}
\label{mainnot}
Throughout the paper, we denote
\noindent\par\noindent
\begin{tabular}{@{\qquad}llllll}
$G_F$       & $=\Gal(\overline{F}/F)$, the absolute Galois group of a field F\cr
$C$         & (hyperelliptic) non-singular projective curve $C$ of genus $g\geqslant 1$\cr
$J_C$       & Jacobian of $C$\cr
$\zeta_m$   & primitive $m^{\rm th}$ root of unity\\[3pt]
\noalign{\noindent In \S3--5, we write}
$K$         & finite extension of $\Q_p$ ($p$ odd)\vphantom{$\int^X$}\cr
$\OK$, $\pi$, $q_K$, $e$  & ring of integers of $K$, uniformiser, size of residue field, ramification degree\cr
$|\cdot|$   & absolute value on $K$, normalised so that $|\pi|=q_K^{-1}$\cr
$\overline{K}, \CK$       & an algebraic closure of $K$ and its completion\cr
$\Gamma$    & Schottky group, see Definition \ref{defschottky}\cr
$s_i, a_i, b_i, c_i, r_i$ & see Construction \ref{conZ}\cr
\end{tabular}
\end{notation}

\section{Reduction to the prime case and $\ell=p$}
First note that for squarefree $m$, the field $\Q(J_C[m])$ is the compositum of $\Q(J_C[p_j])$ for prime divisors $p_j|m$,
so it suffices to prove Theorem \ref{tamethm} when $m=p$ is prime.
In this section, we will reduce the question further to only needing to study the ramification of $\Q(J_C[p])/\Q$ 
at $p$ via a result of Kisin of local constancy of Galois representations in $\ell$-adic families, 
and deal with $p=2$.

\begin{lemma}
\label{tame}
	Let $m=p_1p_2\cdots p_n$, with $p_j$ distinct primes. Let $C/\Q$ be a curve of genus $g$ such that
	\begin{enumerate}
		\item $C$ has semistable reduction at all primes $\ell \leqslant 2g+1$ and $\ell|m$;
		\item $\Q_{p_j}(J_C[p_j]) \cong \Q_{p_j}(\zeta_{p_j})$ for $1 \leqslant j \leqslant n$.
	\end{enumerate}
	Then $\Q(J_C[m])$ is tame.
\end{lemma}

\begin{proof}
	Note $\Q(J_C[m])$ is the compositum of the fields $\Q(J_C[p_j])$ so it suffices to prove that these are all tame. Fix a prime $p=p_j$; we have to show that $\Q(J_C[p])/\Q$ is tamely ramified at $\ell$ for all primes $\ell$; note that by condition (ii), we may assume that $\ell \neq p$.
	
	If $\ell>2g+1$, then a result of Serre--Tate \cite[p.497]{ST68} tells us that the extension is tamely ramified at $\ell$. On the other hand, if $\ell \leqslant 2g+1$, then this follows from Grothendieck's characterisation of inertia on semistable abelian varieties \cite[Proposition 3.5]{Gro72}; see also \cite[Theorem 2.2]{AV11} for a direct proof of this.
\end{proof}

\begin{theorem}
\label{kisin}
Let $\ell$ be a prime. Let $C_f: y^2=f(x)$ be a hyperelliptic curve, with $f\in\Z_{\ell}[x]$ squarefree. For every $m\geqslant 1$, there exists $N\geqslant 1$ such that if $\tilde f\equiv f\mod{\ell^N}$ and $\deg(f)=\deg(\tilde f)$, then $C_{\tilde f}: y^2=\tilde f(x)$ is a hyperelliptic curve with $$J_{C_{\tilde f}}[m] \cong J_{C_f}[m]$$ as $G_{\Q_{\ell}}$-modules.
\end{theorem}

\begin{proof}
This is a special case of \cite[Theorem 5.1(1)]{Kis99}. Note that for $N$ large enough, all $\tilde f\equiv f\mod{\ell^N}$ are squarefree, and so define a $\ell$-adic family of hyperelliptic curves of the same genus.
\end{proof}

With this theorem, we only need to construct genus $g$ hyperelliptic curves $C_{\ell}/\Q_{\ell}$ at each prime $\ell\leqslant 2g+1$ and $\ell|m$ and then glue them together with sufficient congruence conditions in order to realise the $m$-torsion as a tame extension.

To construct a curve $C/\Q_p$ such that $J_C$ is semistable with $\Q_{p}(J_C[p]) \cong \Q_{p}(\zeta_{p})$ we will use the theory of Mumford curves. For simplicity in this approach however, we will assume that $p$ is odd, so we briefly record below a curve that covers the case $p=2$.

\begin{proposition}
\label{prop2}
Let $a_1, \ldots, a_{g+1} \in \Z_2\setminus\{0\}$ have pairwise distinct 2-adic valuations, 
and $0\ne N\in\Z_2$ satisfies $v_2(N) \geqslant \sum_i v_2(a_i)$. 
Let $C/\Q_2$ be the genus $g$ hyperelliptic curve 
$$
  y^2+h(x)y=-N^2, \quad h(x)=\prod\limits_{i=1}^{g+1}(x-a_i).
$$ 
Then the $J_C/\Q_2$ is semistable, and $J_C[2] \subset J_C(\Q_2)$.
\end{proposition}

\begin{proof}
By assumption on the $a_i$, the Newton polygon of $h$ breaks completely, and \cite[Thm 1.2(6,7)]{Dok20} 
shows that $J_C$ is semistable and has totally toric reduction. Next, completing the square and replacing 
$y$ by $y/2$ we see that $C$ is isomorphic to 
$$
  y^2 = (h(x)-2N)(h(x)+2N).
$$
As $v_2(2N)>v_2(h(0))$, the polynomials $h(x)-2N$ and $h(x)+2N$ have the same Newton polygon as $h$, and
so factor completely over $\Z_2$ as well. It follows that $J_C[2]\subset J_C(\Q_2)$. 
\end{proof}

\section{Mumford curves and Whittaker groups}

In this paper, we will only need to concern ourselves with hyperelliptic Mumford curves in which case the Schottky group will be of a particular type called a Whittaker group. For more details on the background of Mumford curves in general, see \cite{GP80}. 

From now on, we suppose that $p \geqslant 3$ and let $K/\Q_p$ be a finite extension. Let $v$ be the normalised valuation on $K$, and $\OK,\pi,q_K,e,|\cdot|$ as in Notation \ref{mainnot}.

\begin{definition}
\label{defschottky}
	Let $\Gamma \!\subset\! \PGL_2(K)$ be a subgroup, acting on $\pp^1(\CK)$ by M\"obius transformations.
	\begin{enumerate}
		\item A point $x \in \pp^1(\CK)$ is a \emph{limit point} of $\Gamma$ if there exists $y \in \pp^1(\CK)$ and an infinite sequence $(\gamma_n) \subset \Gamma$ with $\gamma_n$ distinct and $\lim \gamma_n(y)=x$.
		\item $\Gamma$ is a \emph{Schottky group} if it is discrete, free, and finitely generated.
		\item Suppose $\Gamma$ is Schottky. Let $\Omega_{\Gamma} = \pp^1(\CK) - \{\text{limit points of } \Gamma\}$. Then $\Omega_{\Gamma}/\Gamma$ is a \emph{Mumford curve} of genus equal to the rank of $\Gamma$.
		\item Let $\Gamma$ be a Schottky group. If the associated Mumford curve is hyperelliptic, then $\Gamma$ 
		is called a \emph{Whittaker group}.
	\end{enumerate}
\end{definition}

\begin{example}
	Let $q \in K$ be such that $|q|<1$ and let $\gamma=\left( \begin{smallmatrix} q & 0 \\ 0 & 1 \end{smallmatrix} \right)$. Then $\Gamma=\langle \gamma \rangle$ is Schottky of rank~$1$. Moreover, $\Omega_{\Gamma}=\pp^1(\CK)-\{0, \infty\}$, and the corresponding Mumford curve is isomorphic to the Tate curve associated to $q$.
\end{example}

The construction of Mumford curves is analytic, so computing an algebraic model for them is in general difficult. The main approach is to construct a good fundamental domain, which we do via $p$-adic discs; we briefly set up some notation for them. \\

\noindent \textbf{Notation.} Let $c,r \in \CK$. We define the \emph{open disc} $B(c,r)$ and \emph{closed disc} 
$\overline{B(c,r)}$ with centre $c$ and radius $r$ as $$B(c,r)=\{ z \in \CK : |z-c|<|r|\}, \qquad \overline{B(c,r)}=\{ z \in \CK : |z-c| \leqslant |r|\}.$$

\begin{definition}
\label{fundG}
	Let $\Gamma$ be a Schottky group of rank $g$. Then a set $F$ is called a 
	\emph{good fundamental domain} for $\Gamma$ if:
	\begin{enumerate}
		\item $F=\pp^1(\CK) - \bigl(\bigcup_{i=1}^g (B_i \cup B'_i)\bigr)$ where $B_1,B'_1,\cdots,B_g,B'_g$ are $2g$ open discs with centres~in~$K$;
		\item The closed discs $\overline{B_1},\overline{B'_1},\cdots,\overline{B_g},\overline{B'_g}$ are disjoint;
		\item $\Gamma$ is generated by elements $\gamma_1, \cdots, \gamma_g$ such that $\gamma_i(\pp^1(\CK) - B_i)=\overline{B'_i}$ and $\gamma_i(\pp^1(\CK) - \overline{B_i})=B'_i$ for $1 \leqslant i \leqslant g$.
	\end{enumerate}
	In this case, we say that the generators $\gamma_1,\ldots,\gamma_g$ are \emph{in good position}.
\end{definition}

\begin{proposition}
	Every Schottky group has a good fundamental domain. Conversely, given a set $F$ satisfying conditions (i) and (ii), there exists a Schottky group with good fundamental domain $F$.
\end{proposition}

\begin{proof}
	See \cite[I.4.1.3 and I.4.1.4]{GP80}.
\end{proof}

For hyperelliptic Mumford curves, one constructs a Whittaker group via a suitable choice of $2g+2$ points as follows: \\

\begin{construction}
\label{conZ}
\leavevmode
\begin{enumerate}
\item Let $Z=\{a_1,b_1,\ldots,a_g,b_g,a_{\infty}\!=\!1,b_{\infty}\!=\!\infty \}$ be a set of $2g+2$ distinct points of~$\pp^1(K)$.
\item For each pair $a_i,b_i$, let $c_i=\tfrac{a_i+b_i}{2}$, $r_i=\tfrac{b_i-a_i}{2}$ if $i \leqslant g$.
\item Let $s_i=\begin{psmallmatrix} c_i & r_i^2-c_i^2 \\ 1 & -c_i \end{psmallmatrix}$ if $i \leqslant g$ and $s_{\infty}=\begin{psmallmatrix} 1 & -2 \\ 0 & -1 \end{psmallmatrix}$  be involutions in $\PGL_2(K)$ fixing $a_i,b_i$.
\item Let $\Gamma_Z=\langle s_1s_{\infty},\cdots, s_gs_{\infty} \rangle$.
\end{enumerate}
\end{construction}

\begin{definition}
	Let $Z$ be a set of $2g+2$ distinct points. If the corresponding group $\Gamma_Z$ is a Whittaker group of rank $g$, then we say that $Z$ is in good position.
\end{definition}

\begin{remark}
\leavevmode
\begin{enumerate}
	\item To check if $\Gamma_Z$ is Schottky, let $B_i=B(c_i,r_i)$ and $B'_i=s_{\infty}(B_i)$  for $i \leqslant g$. Then it suffices to check the conditions (i) and (ii) of Definition \ref{fundG} for these $2g$ discs to see if they define a good fundamental domain.
	\item We can always suppose that $0,1,\infty \in Z$ by applying a M\"obius transformation. 
  This changes $\Gamma_Z$ to a conjugate subgroup and gives an isomorphic Mumford curve.
	\item Note that the construction requires a choice of pairing on $Z$. One can show that there is at most one pairing on $Z$ such that it is in good position.
  \item The equation of the hyperelliptic curve $C=\Omega_{\Gamma_Z}/\Gamma_Z$ is (see {\cite[p.279]{GP80}}) 
	$$
	  C\colon y^2=\prod\limits_{z \in Z} (x-\theta(0,1;z)),
	$$ where {$\theta(0,1;z)=\!\!\prod\nolimits_{w\in W_Z} \tfrac{z-w(0)}{z-w(1)}$} and $W_Z=\langle s_1, \ldots, s_g, s_{\infty} \rangle$ is the group generated by the associated involutions;
  the 2:1 map $C\to \pp^1$ is $\Omega_{\Gamma_Z}/\Gamma_Z\to \Omega_{\Gamma_Z}/W_Z$.
\end{enumerate}
	
\end{remark}

\begin{lemma}[{\cite[Lemma 5.5, Theorem 5.7]{Kad07}}]
\label{goodpos}
	Let $Z=\{a_1=0,b_1,a_2,\ldots,a_g,b_g,a_{\infty}=1$, $b_{\infty}=\infty\}$ be a set of $2g+2$ distinct points. Suppose
	\begin{itemize}
		\item $0<|b_1|<|a_2| \leqslant |b_2| \leqslant |a_3| \cdots \leqslant |b_g| < 1$;
		\item $\tfrac{|r_i|}{|c_i-c_j|}<1$ for all distinct $1 \leqslant i,j \leqslant g$.
	\end{itemize}
	Then:
	\begin{enumerate}
		\item The points of $Z$ are in good position;
		\item $\Gamma=\langle s_1s_{\infty}, \cdots, s_gs_{\infty} \rangle$ is a Whittaker group of rank $g$;
		\item A good fundamental domain for $\Gamma$ is given by the complement of the discs $B_i=B(c_i,r_i)$ and $B'_i=B(2-c_i,r_i)$, $1 \leqslant i \leqslant g$.
	\end{enumerate}
\end{lemma}

We will now implicitly assume these assumptions in the lemma whenever we deal with a Whittaker group. For two discs $B,B'$, we denote by $d(B,B')$ the corresponding metric coming from the standard one on the Berkovich line $\pp^{1,an}$ (see for example \cite[p.7]{MR15}).

\begin{lemma}
\leavevmode
	\begin{enumerate}
		\item Let $i \neq j$. Then $d(B_i,B_j)=d(B'_i,B'_j)=\log_p \frac{|c_i-c_j|^2}{|r_ir_j|}$.
		\item For all $i,j$, $d(B_i,B'_j)=\log_p \frac{1}{|r_ir_j|}$.
	\end{enumerate}
	In particular, the minimum distance, $m_{\Gamma}$\footnote{This depends on the choice of a good fundamental domain and not just $\Gamma$.}, between two distinct discs is $\min\limits_{i \neq j \leqslant g} \log_p \frac{|c_i-c_j|^2}{|r_ir_j|}$.
\end{lemma}

\begin{proof}
For the first part, note that the smallest disc containing $B_i$ and $B_j$ is $B(c_i,|c_i-c_j|)$ and the statement follows from the definition. For the second part, we get $d(B_i,B'_j)=\tfrac{|2-c_i-c_j|^2}{|r_ir_j|}$ and note that the numerator is a unit as $p \neq 2$ and the centres $c_i$ are integral non-units. The minimum now follows.
\end{proof}

\section{Approximation of the period matrix}

Let $\Gamma$ be a Schottky group with generators $\gamma_1, \ldots, \gamma_g$ in good position. Let $B_1, \ldots, B_g$, $B'_1, \ldots, B'_g$ be the associated disjoint discs defining the fundamental domain such that $\gamma_k(\pp^1(\CK) - B'_k)=\overline{B_k}$ for all $k$. We define the closure of an open disc $B$ as $\overline{B}$, the boundary of $B$ to be $\partial B:=\overline{B} \setminus B$ and the diameter of $B$ as $\diam(B)=\sup\limits_{x,y \in \overline{B}} |x-y|$.

\begin{notation}
	For a free group $\Gamma=\langle \gamma_1, \gamma_2, \cdots, \gamma_g \rangle$, we let $\Gamma_n$ be the subset consisting of all elements of $\Gamma$ of reduced word length at most $n$.
\end{notation}

For $\id \neq \gamma \in \Gamma_1$, we define $\Bg= \begin{cases}
	B_k \qquad \text{ if $\gamma=\gamma_k,$ \, \quad $k=1,\ldots,g$}; \\
	B'_k \qquad \text{ if $\gamma=\gamma_k^{-1},$ \quad $k=1,\ldots,g$}.
\end{cases}$

\begin{lemma}
\label{approx1}
	Let $a \in \partial B'_i, z \in \partial B'_j$. Let $\id \neq \gamma \in \Gamma_1$. Let $z_j,z_j',\zg$ be centres of $B_j,B'_j,\Bg$ respectively.
	\begin{enumerate}
	\item[]
	\item If $\gamma \neq \gamma_j^{-1}$, then $\left| \dfrac{z-\gamma a}{z-\gamma\gamma_ia}-1 \right| \leqslant \dfrac{\diam(\overline{\Bg})}{|z_j'-\zg|}$;
	\item If $\gamma \neq \gamma_j$, then $\left| \dfrac{\gamma_jz-\gamma a}{\gamma_jz-\gamma\gamma_ia}-1 \right| \leqslant \dfrac{\diam(\overline{\Bg})}{|z_j-\zg|}$.
	\end{enumerate}
\end{lemma}

\begin{proof}
	We prove the first part; the second part is analogous. First note that $\frac{z-\gamma a}{z-\gamma \gamma_ia}-1=\frac{\gamma \gamma_ia-\gamma a}{z-\gamma \gamma_ia}$. Since $a \in \partial B'_i$, we have that $\gamma a, \gamma \gamma_ia \in \overline{\Bg}$ and hence $|\gamma \gamma_ia - \gamma a| \leqslant \diam(\overline{\Bg})$.
	
	Since $\gamma \neq \gamma_j^{-1}$, the discs $B'_j$ and $\Bg$ are disjoint, so let $z_j', \zg$ be centres of $B'_j,\Bg$ respectively. Now $z,z_j' \in \overline{B'_j}$ and $\zg \notin\overline{B'_j}$, so $|z_j'-\zg|>|z-z_j'|$. Similarly $|z_j'-\zg|>|\zg-\gamma \gamma_ia|$ using $\Bg$. Hence 
	$$
		|z-\gamma \gamma_ia| = |z-z_j'+z_j'-\zg+\zg-\gamma\gamma_ia| = |z_j'-\zg|,
  $$ 
  by the ultrametric triangle inequality.
\end{proof}

We now return to the case where $\Gamma$ is a Whittaker group and continue our notation from~\S 3.
The Jacobian $J_{\Omega_\Gamma/\Gamma}$ has a $g\times g$ \emph{period matrix} $Q=(Q_{ij})$ 
whose entries can be computed in terms of $\Gamma$ (see \cite{GP80} VI.2)
%
$$
  Q_{ij}= \prod\limits_{\gamma \in \Gamma} \dfrac{(z-\gamma a)(\gamma_jz-\gamma\gamma_i a)}
     {(z-\gamma\gamma_ia)(\gamma_jz- \gamma a)},
$$ 
for any choice of non-conjugate ordinary points $a,z$.

\begin{notation}
For a subset $S \subset \Gamma$, let 
$$
  Q_{ij}^S=\prod\limits_{\gamma \in S} \dfrac{(z-\gamma a)(\gamma_jz-\gamma\gamma_i a)}{(z-\gamma\gamma_ia)(\gamma_jz- \gamma a)}.
$$ 
If $S=\Gamma_n$, we write $Q_{ij}^n$ for $Q_{ij}^{\Gamma_n}$; 
clearly $\lim_{n\to\infty} Q_{ij}^n = Q_{ij}$.
\end{notation}

\begin{lemma}
\label{approx}
	Let $q \in K$ be such that $\max\limits_{i \neq j} \dfrac{|r_i|}{|c_i-c_j|}\leqslant |q|<1$. Then 
	$\Bigl| \dfrac{Q_{ij}^1}{Q_{ij}}-1 \Bigr| < |q|.$
\end{lemma}

\begin{proof}
	We have $\Bigl| \tfrac{Q_{ij}^1}{Q_{ij}}-1 \Bigr| \leqslant q_K^{-em_{\Gamma}}$ by \cite[Theorem 3.6]{MR15}\footnote{Note that under our normalisation $|p|=q_K^{-e}$ in contrast to \cite{MR15} who use $|p|=p^{-1}$.},
	and $\max\limits_{i \neq j} \tfrac{|r_i|}{|c_i-c_j|}> \max\limits_{i \neq j} \tfrac{|r_ir_j|}{|c_i-c_j|^2}=q_K^{-em_{\Gamma}}$ since the discs are disjoint.
\end{proof}

\begin{theorem}
\label{redid}
	Let $q \in K$ be such that $\max\limits_{i \neq j} \frac{|r_i|}{|c_i-c_j|}\leqslant |q|<1$. Let $a \in \partial B'_i$, $z \in \partial B'_j$ be distinct mod $\Gamma$. Define $$Q_{ij}^{\alpha}=Q_{ij}^{0}\dfrac{(z-\gamma_j^{-1}a)(\gamma_jz-\gamma_ja)}{(z-\gamma_j^{-1}\gamma_ia)(\gamma_jz-\gamma_j\gamma_ia)}.$$ Then $$\left| \dfrac{Q_{ij}^{\alpha}}{Q_{ij}}-1 \right| \leqslant |q|.$$
\end{theorem}

\begin{proof}
	Note first that such a $q$ exists by Lemma \ref{goodpos}. Using Lemma \ref{approx}, we only need to consider the contributions from non-identity elements in $\Gamma_1$. The result is then immediate from Lemma \ref{approx1}.
\end{proof}


We will now compute $Q_{ij}^{\alpha}$ to get an explicit formula.
By choosing the auxiliary parameters $a,z$ carefully, we will not need to distinguish between the cases $i=j$ and 
$i \neq j$, and we find that $Q_{ij}^{\alpha}=Q_{ij}^{0}$ with this choice.

\begin{lemma}
\label{lembndry}
	Let $a \in \partial B'_i$, $z \in \partial B'_j$ and assume $a\neq z$ if $i=j$. Then $a \not\equiv z \mod{\Gamma}.$
\end{lemma}

\begin{proof}
	We shall adapt the proof of \cite[Lemma 2.4]{MR15}. In fact, we shall prove that $\gamma a$ is contained in the interior of the open disc $B_{h_1}$ (continuing notation from above), where $\gamma=h_1\cdots h_m$ as a reduced word, unless $\gamma \in \{\id,\gamma_i\}$. Note that $\gamma_k(\pp^1 \setminus \overline{B'_k})=B_k$ and moreover $\gamma_k(\partial B'_k)=\partial B_k$ for all $k$. 
	
	If $h_m \neq \gamma_i$ then $h_ma \in B_{h_m}$ and hence iteratively we have $\gamma a \in B_{h_1}$. If $h_m=\gamma_i$, then $h_ma \in \partial B_i$ so if $m \geqslant 2$, then $h_{m-1} \neq \gamma_i^{-1}$ so proceeding similarly we have $\gamma a \in B_{h_1}$. Moreover, note $\gamma_i a \neq z$ since $B'_j \neq B_i$. Lastly, if $\gamma=\id$, then $a \neq z$ by assumption.
\end{proof}

\begin{lemma}
\label{alphaid}
Let $a=2-c_i+r_i$ and $z=2-c_j-r_j$. Then
\begin{enumerate}
\item $a\in \partial B'_i$, $z\in \partial B'_j$, and $a$ and $z$ are distinct mod $\Gamma$.
\item $Q_{ij}^{0}=\bigl(\tfrac{c_i-c_j-r_i-r_j}{2-c_i-c_j+r_i-r_j}\bigr)^2$ for all $1 \!\leqslant\! i,j \!\leqslant\!g$.
\item $Q_{ii}^{0}=\bigl( \tfrac{r_i}{c_i-1}\bigr)^2$ for all $1 \!\leqslant\! i \!\leqslant\!g$.
\item $Q_{ij}^{\alpha}=Q_{ij}^{0}$.
\end{enumerate}
\end{lemma}

\begin{proof}
\begin{enumerate}
\item[]
\item Follows from Lemma \ref{lembndry}.
\item We have $\gamma_k=\begin{psmallmatrix}
	c_k & c_k^2-r_k^2-2c_k \\
	1 & c_k-2
	\end{psmallmatrix}$ for all $k$. 
	From this we compute explicitly that $\gamma_ia=c_i-r_i$ and similarly $\gamma_jz=c_j+r_j$. 
Now the claim follows from 
$$
\begin{array}{lll@{\qquad\quad}lll}
  z\!-\!a &=& (c_i\!-\!c_j)\!-\!(r_i\!+\!r_j), & z\!-\!\gamma_ia &=& 2\!-\!c_i\!-\!c_j\!+\!r_i\!-\!r_j,\cr
  \gamma_jz \!-\!\gamma_ia &=& -(z\!-\!a), & \gamma_jz \!-\!a &=& -(z\!-\!\gamma_i a).
\end{array}
$$
\item This follows from (2), by setting $i=j$.
\item We compute ${(z-\gamma_j^{-1}a)(\gamma_jz-\gamma_ja)}/{(z-\gamma_j^{-1}\gamma_ia)(\gamma_jz-\gamma_j\gamma_ia)}$. The claim now follows from
$$
\begin{array}{lll@{\qquad\quad}lll}
		z-\gamma_j^{-1}a &=& -r_j - \frac{r_j^2}{-2+c_i+c_j-r_i}, &
		z-\gamma_j^{-1}\gamma_ia &=& -r_j + \frac{r_j^2}{c_i-c_j-r_i}, \cr
		\gamma_jz-\gamma_j\gamma_ia &=& -(z-\gamma_j^{-1}a), &
		\gamma_jz-\gamma_ja &=& -(z-\gamma_j^{-1}\gamma_ia).
\end{array}$$
\end{enumerate}
\end{proof}

\section{Tame torsion}

\begin{lemma}
\label{power}
	Let $a \in 1+\pi^N\OK$ for some positive integer $N$. If $em \leqslant N$, then $x^m-a$ has a root in $\OK$. In particular, every element of $1+\pi^{em}\OK$ is an $m^{\text{th}}$ power for all $m \geqslant 1$.
\end{lemma}

\begin{proof}
	This is a simple application of Hensel's lemma, where we use the version that states there is a lift of a root $a_0$ (in the residue field) of a polynomial $f$ if $v(f(a_0))>2v(f'(a_0))$, where we use $f=x^m-a$ and $a_0=1$.
	
	Note that $v(f(a_0)) \geqslant N$ by construction and $f'(a_0)=m$, so $v(f'(a_0))=ev_p(m)$ where $v_p$ is the standard $p$-adic valuation on $\Z$. Now \begin{eqnarray*}
		v_p(m) &\leqslant& \log_p(m), \\
			   &<& \ln(m) \qquad \text{ as $p\geqslant 3$,} \\
			   &\leqslant& \frac{m}{2}  \quad \qquad \text{ by bounds on $\ln$},
	\end{eqnarray*}
	so $v(f'(a_0))<\frac{em}{2}$ and the result follows.
\end{proof}

\begin{lemma}
\label{iipow}
	Let $r_i\!=\!\pi^{em\alpha}$, $c_i\!=\!2\pi^{em\beta}$ for some $\alpha, \beta\!>\!0$. 
	Then $(\frac{r_i}{1-c_i})^2$ is an $m^{\text{th}}$ power in~$\OK$.
\end{lemma}

\begin{proof}
	We have $(1-c_i)^2 \in 1+\pi^{em}\OK$, so it is an $m^{\text{th}}$ power by Lemma \ref{power}.
\end{proof}

\begin{lemma}
\label{ijpow}
	Let $r_i\!=\!\pi^{em\alpha_i}, r_j\!=\!\pi^{em\alpha_j}, c_i\!=\!2\pi^{em\beta_i}, c_j\!=\!2\pi^{em\beta_j}$ 
	with $\alpha_i,\alpha_j,\beta_i,\beta_j$ distinct positive integers with 
	$\beta_i,\beta_j<\alpha_i,\alpha_j$. 
	Then $\bigl(\tfrac{c_i-c_j-r_i-r_j}{2-c_i-c_j+r_i-r_j}\bigr)^2$ is an $m^{\text{th}}$ power in~$\OK$.
\end{lemma}

\begin{proof}
	Without loss of generality, suppose $\beta_i<\beta_j$. Then 
	$$
	  c_i-c_j-r_i-r_j=2\pi^{em\beta_i}
	  \bigl( 1-\pi^{em(\beta_j-\beta_i)}-\frac{1}{2}\pi^{em(\alpha_i-\beta_i)}-\frac{1}{2}\pi^{em(\alpha_j-\beta_i)}\bigr)
    \in 2\pi^{em\beta_i}(1+\pi^{em}\OK),
	$$ 
	which is twice an $m^{\text{th}}$ power by Lemma \ref{power}.
	On the other hand, the denominator is 
	$$
	  2-c_i-c_j+r_i-r_j=2
	  \bigl(1-\pi^{em\beta_i}\bigl(1+\pi^{em(\beta_j-\beta_i)}+\frac{1}{2}\pi^{em(\alpha_i-\beta_i)}-\frac{1}{2}\pi^{em(\alpha_j-\beta_i)}\bigr)\bigr)
	  \in 2(1+\pi^{em}\OK),
	  $$ 
	  which is also twice an $m^{\text{th}}$ power.
\end{proof}

\begin{theorem}
\label{mainthm}
	Let $m\geqslant 1$, and
	\begin{itemize}
		\item $\alpha_1>\alpha_2>\cdots >\alpha_g>\beta_2>\beta_3>\cdots >\beta_g$ positive integers;
		\item $r_1=c_1=\pi^{em\alpha_1}$, and $r_i=\pi^{em\alpha_i}, \,\, c_i=2\pi^{em\beta_i}$ for $2 \leqslant i \leqslant g$;
		\item $a_i=c_i-r_i, \,\, b_i=c_i+r_i$ for $1 \leqslant i \leqslant g$.
	\end{itemize}
	
	Then:
	\begin{enumerate}
		\item $a_1=0$;
		\item $0<|b_1|<|a_2| \leqslant |b_2| \leqslant |a_3| \cdots \leqslant |b_g| < 1$;
		\item $\dfrac{|r_i|}{|c_i-c_j|} \leqslant q_K^{-em} < 1$ for all distinct $1 \leqslant i,j \leqslant g$;
		\item $\bigl(\frac{c_i-c_j-r_i-r_j}{2-c_i-c_j+r_i-r_j}\bigr)^2$ is an $m^{\text{th}}$ power in $\OK$ for all $1 \leqslant i,j \leqslant g$.
		\item Let $Q=(Q_{ij})$ denote the period matrix of the corresponding abelian variety. Then $Q_{ij}$ is an $m^{\text{th}}$ power for all $1 \leqslant i,j, \leqslant g$.
	\end{enumerate}
\end{theorem}

\begin{proof}
\begin{enumerate}
\item[]
\item Note $a_1=c_1-r_1=0$ by definition.
\item Observe that for $i \geqslant 2$, $|a_i|=|b_i|=|c_i|=q_K^{-em\beta_i}$. Since the $\beta_i$ are decreasing and $a_i,b_i \in \pi\OK$, we have $|a_2| \leqslant |b_2| \leqslant |a_3| \cdots \leqslant |b_g| < 1$. Lastly note $|b_1|=|2\pi^{em\alpha_1}|<|a_2|$.
\item We compute that for $i \neq j$, $\dfrac{|r_i|}{|c_i-c_j|}=\dfrac{|\pi^{em\alpha_i}|}{|\pi^{em\beta_j}|}=q_K^{-em(\alpha_i-\beta_j)}$ where we suppose $i<j$ without loss of generality. Since $\alpha_i>\beta_j$, we are done.
\item First suppose $i=j$. If $i \neq 1$, then this follows directly from Lemma \ref{iipow}; the same proof also works for $i=1$. Now suppose $i \neq j$. If $i,j \geqslant 2$, then this is Lemma \ref{ijpow}. If $i=1$ or $j=1$, then one can apply the same proof using the simplification $c_1=r_1$.
\item By (iv), we have that $Q_{ij}^{0}$ is an $m^{\text{th}}$ power. Now by Theorem \ref{redid}, Lemma \ref{alphaid}(4) and (iii), $\bigl| \frac{Q_{ij}^{0}}{Q_{ij}}-1 \bigr| \leqslant q_K^{-em}$ hence $Q_{ij}^{0}=Q_{ij}(1+\pi^{em}b)$ for some $b \in \OK$. Since $Q_{ij}^{0}$ and $1+\pi^{em}b$ are $m^{\text{th}}$ powers (by Lemma \ref{power}), so is $Q_{ij}$.
\end{enumerate}
\end{proof}

\begin{lemma}
\label{periodpow}
	Let $J/K$ be an abelian variety with a Raynaud parameterisation 
	$J \cong \bigl(\overline{K}^{\times}\bigr)^g/Q$. Let $m>1$ and suppose every entry in the period matrix $Q$ is an $m^{\text{th}}$ power in~$K$.~Then 
	$$
	  J[m] \cong \mu_m^g \times \left( \mathbb{Z}/m\mathbb{Z}\right)^g
	$$ 
	as $G_K$-modules. Here $\mathbb{Z}/m\mathbb{Z}$ has a trivial action, and 
	$\mu_m=\langle\zeta_m\rangle\subset\overline{K}$ is the set of $m^{\text{th}}$ roots of unity, with natural action.	
  In particular, $K(J[m])=K(\zeta_m)$.
\end{lemma}

\begin{proof}
	Recall that $J(\overline K) \cong \bigl(\overline{K}^{\times}\bigr)^g/Q$ as $G_K$-modules. Let $Q=(Q_{ij})$. Then 
	$$
	  J[m] = \mu_m^g \times \langle  (Q_{i1}^{1/m},Q_{i2}^{1/m},\cdots, Q_{ig}^{1/m}), \, \, i=1,\ldots, g \rangle.
	$$
  As every $Q_{ij}\in K^\times$ is an $m^{\text{th}}$ power, the result follows.
\end{proof}


\begin{theorem}
\label{tamepcons}
	Fix an integer $g \geqslant 1$. Let 
	\begin{itemize}
		\item $\alpha_i=2g-i$ for $1\leqslant i \leqslant g$, and $\beta_i=g-i+1$ for $2 \leqslant i \leqslant g$;
		\item $r_1=c_1=\pi^{p\alpha_1}$, and $r_i=\pi^{p\alpha_i}, c_i=2\pi^{p\beta_i}$ for $2 \leqslant i \leqslant g$.
	\end{itemize}
	Let $C/K$ be the corresponding genus $g$ hyperelliptic Mumford curve given by Construction \ref{conZ}. Then $J_C$ is semistable and $K(J_C[p])=K(\zeta_p)$.
\end{theorem}

\begin{proof}
	Recall that all Mumford curves are semistable (see for example \cite[Theorem 2.12.2]{GP80}). By Theorem \ref{mainthm} with $m=p$, every entry of the period matrix of $J_C$ is a $p^{th}$ power; the statement now follows from Lemma \ref{periodpow} with $m=p$.
\end{proof}

\begin{theorem}
\label{tamecor}
	Fix a positive integer $g$ and squarefree integer $m$. Then there exists a non-singular projective 
	curve $C/\Q$ of genus $g$ such that $\Q(J_C[m])$ is tame.
\end{theorem}

\begin{proof}
	By Kisin's result (Theorem \ref{kisin}) we need only choose a suitable genus $g$ hyperelliptic curve $C_{\ell}$ for the finite set of primes $\ell \leqslant 2g+1$ and $\ell\mid m$; if $\ell \nmid m$ we take $C_{\ell}$ to be semistable at $\ell$ (e.g. good reduction at $\ell$). For $\ell \mid m$, $\ell \neq 2$, Theorem \ref{tamepcons} 
	with $K=\Q_{\ell}$, $p=\ell$ 
	provides a construction of a genus $g$ hyperelliptic curve $C_{\ell}$ such that $\Q_{\ell}(J_{C_{\ell}}[\ell])\cong \Q_{\ell}(\zeta_{\ell})$; similarly we can use Proposition \ref{prop2} if $\ell=2$. We are now done by Lemma \ref{tame}.
\end{proof}

\begin{remark}
	The same approach works to construct a curve $C/\Q_p$ such that $\Q_p(J_C[p^n])=\Q_p(\zeta_{p^n})$ for any $n \geqslant 1$ but note that this is wildly ramified if $n\neq 1$. However we can give global curves $C/\Q$ such that $\Q(J_C[m])/\Q(\zeta_{m})$ is a tame extension any odd integer $m$.
\end{remark}

\section{The tame inverse Galois problem}

In this section, we investigate the tame version when $G$ is of the form $\GSp_{2g}(\fp)$ via the mod $p$ representation of abelian varieties.

\begin{remark}
	An alternative approach to force surjectivity is to ensure $\End A = \mathbb{Z}$ (to guarantee this, take $\Gal(f) \cong S_{\deg(f)}$ and apply \cite[Theorem 2.1]{Zar00}) and then apply Serre's open image theorem to obtain surjectivity for $p$ sufficiently large\footnote{This is sufficient if $\dim A$ is odd \cite[Corollaire p.51]{SerIV}; otherwise we need an extra local condition due to the Mumford-Tate group \cite[Theorem 1]{Hal11}.}. There are two problems with this however: we do not know precisely what sufficiently large means and more importantly this says nothing for small $p$.
\end{remark}

\begin{conj}[Goldbach + $\varepsilon$]
\label{Gold}
	Let $n\geqslant 4$ be an even integer. Then there exists primes $q_1,q_2,q_3$ such that $q_1 \leqslant q_2 < q_3< n$ and $q_1+q_2=n$. We refer to $(q_1,q_2,q_3)$ as a \emph{Goldbach triple~for~$n$}.
\end{conj}

\begin{conj}[Double Goldbach + $\varepsilon$]
\label{2Gold}
	Let $n$ be a positive even integer. Then there exists primes $q_1,q_2,q_3,q_4,q_5$ such that $q_4<q_1 \leqslant q_2<q_5 < q_3< n$ and $q_1+q_2=q_4+q_5=n$.
\end{conj}

\begin{theorem}
\label{surjmodp}
	Let $p\geqslant 5$ be prime and let $A/\Q$ be a principally polarised abelian variety of dimension $g$. Suppose:
	\begin{enumerate}
		\item The $G_{\Q}$-action on $A[p]$ is irreducible, primitive and contains a transvection;
		\item $\Q_p(A[p]) \cong \Q_p(\zeta_p)$;
		\item $A$ is semistable at $\ell$ for all primes $\ell \leqslant 2g+1$.
	\end{enumerate}
	Then $\Gal(\Q(A[p])/\Q)\cong \GSp_{2g}(\fp)$ and $\Q(A[p])$ is tame. 
	The same holds for $p=3$ if $A[3] \otimes_{\mathbb{F}_3} \overline{\mathbb{F}}_3$ is irreducible and primitive.
\end{theorem}

\begin{proof}
	By \cite[Theorem 5.3]{AD20}, condition $(i)$ implies that $\Gal(\Q(A[p])/\Q)\cong \GSp_{2g}(\fp)$
	(including $p=3$).
	The claim that $\Q(A[p])$ is tame follows from Lemma \ref{tame}.
\end{proof}

Before we state an explicit version of the above theorem, we need some quick definitions.

\begin{definition}
	Let $p$ be a prime and let $f(x)= x^m+a_{m-1}x^{m-1}+ \cdots +a_0 \in \mathbb{Z}_p[x]$ be a squarefree monic polynomial. Fix an integer $t\geqslant 1$.
	\begin{enumerate}
		\item We say that $f$ is $t$-Eisenstein at $p$ if $v_p(a_i) \geqslant t$ for all $i$ and $v_p(a_0)=t$.
		\item Let $q_1,\cdots,q_k$ be rational primes. We say that $f$ is of type $t\!-\!\{q_1,\cdots,q_k\}$ 
		if it can be factored over $\mathbb{Z}_p[x]$ as 
		$$
		  f(x)=h(x)\prod\limits_{i=1}^k g_i(x-\alpha_i),
		$$ 
		for some $\alpha_i \in \mathbb{Z}_p$ such that $\alpha_i \not\equiv \alpha_j \mod{p}$ for $i \neq j$, $g_i(x)$ is $t$-Eisenstein of degree $q_i$ and the reduction mod $p$, $\overline{h}(x)$, of $h(x)$ is separable with $\overline{h(\alpha_i)}\neq 0$ for all $i$.
	\end{enumerate}
\end{definition}

\begin{theorem}
\label{tamemodp}
	Let $C/\Q:y^2=f(x)$ be a hyperelliptic curve of genus $g$ and Jacobian $J_C$. Assume $2g+2$ satisfies Conjecture \ref{Gold} and let $(q_1,q_2,q_3)$ be a Goldbach triple. Fix an odd prime $p \neq q_1,q_2,q_3$. 
	
	Choose primes $p_1,p_2,p_3>\max (2g+1,p)$ such that:
	\begin{itemize}
		\item $p_2$ is a primitive root modulo $q_1$ and modulo $q_2$; 
		\item $p_3$ is a primitive root modulo $q_3$;
		\item If $p=3$, then moreover suppose that $p_2 \equiv p_3 \equiv 1 \mod{3}$.
	\end{itemize}
	
	Suppose:
	\begin{enumerate}
		\item $f(x)$ has type $1-\{2\}$ at $p_1$;
		\item $f(x)$ has type $1-\{q_1,q_2\}$ at $p_2$;
		\item $f(x)$ has type $2-\{q_3\}$ at $p_3$;
		\item $J_C$ is semistable at all $\ell \notin \{p_2,p_3\}$;
		\item $J_C$ is totally toric at $p$;
		\item $\Q_p(J_C[p]) \cong \Q_p(\zeta_p)$.
	\end{enumerate}
	
	Then $\Gal(\Q(J_C[p])/\Q)\cong \GSp_{2g}(\fp)$ and $\Q(J_C[p])/\Q$ is tame.
\end{theorem}

\begin{proof}
	This is a slight reformulation of \cite[Theorem 6.2]{AD20} where we can weaken some of the hypotheses since $p$ is fixed.
	
	Suppose first that $p\geqslant 5$. Then condition $(i)$  implies the existence of a transvection \cite[Lemma 2.9]{AD20}, whereas $(ii)$ and $(iii)$ imply that $J_C[p]$ is irreducible \cite[Lemma 3.2]{AD20}. Primitivity follows from $(iv)$ and $(v)$ (cf. \cite[Remark 6.1]{AD20}); the result now follows from Theorem \ref{surjmodp}. For the case $p=3$, the same argument as \cite[Theorem 6.5]{AD20} holds.
\end{proof}

\begin{corollary}
\label{IGPcor}
	Fix a positive integer $g$ and assume $2g+2$ satisfies Conjecture \ref{Gold}. Fix an odd prime $p$. If there exists a Goldbach triple for $2g+2$ not containing $p$, then there exists a curve $C/\Q$ of genus $g$ such that $\Gal(\Q(J_C[p])/\Q) \cong \GSp_{2g}(\fp)$ and $\Q(J_C[p])$ is tame.
\end{corollary}

\begin{remark}
	If $2g+2$ satisfies Double Goldbach as well (Conjecture \ref{2Gold}), then the conclusion holds for all odd primes $p$ by applying the statement with the Goldbach triples $(q_4,q_5,q_3)$ and $(q_1,q_2,q_5)$. Double Goldbach has been numerically verified by Anni--Dokchitser (cf. \cite[Remark 6.6]{AD20}) to hold for all $g \leqslant 10^7,$ excepting $g=1,2,3,4,5,7,13$.
\end{remark}

\begin{remark}
\label{2g+1}
	Observe that if $q=2g+1$ is prime, then we do not need to use a Goldbach triple; imposing that $f(x)$ has type $1-\{q\}$ at some large prime ensures that $J_C[p]$ is an irreducible $G_{\Q}$-representation and we only then need to avoid $p=q$ for the same result.
\end{remark}

Combining the above results with those of Arias-de-Reyna--Vila for $g\leqslant 2$ 
(\cite[Theorem 1.2]{AV09},\cite[Theorem 5.3]{AV11}) and Remark \ref{2g+1}, we find that the remaining cases for odd $p$ and small genus are hence as follows:
\begin{center}
\begin{tabular}{c|c}
	Genus & primes excluded \\ \hline
	$3$ & $7$ \\
	$4$ & $5,7$ \\
	$5$ & $11$ \\
	$7$ & $5,11,13$ \\
	$13$ & $11,17$.
\end{tabular}
\end{center}

\noindent
The reason for these exceptions is that the method of Anni--Dokchitser uses a Goldbach triple to ensure that $J_C[p]$ is an irreducible $G_\Q$-module when $p$ is not in the Goldbach triple. Instead, we take a different approach to ensure that the mod $p$ representation is surjective.

\begin{lemma}
\label{smallIGP}
	Let $C/\Q:y^2=f(x)$ be a hyperelliptic curve of genus $g$. 
	Let $$\rho\colon\Gal(\Q(J_C[p])/\Q) \rightarrow \GSp_{2g}(\fp)$$ be the mod $p$ representation of $J_C$. Suppose that
	\begin{enumerate}
		\item $f$ has type $1\!-\!\{2\}$ at some prime $p_1$;
		\item For some prime $\ell \neq p$ of good reduction for $C$, the reduction mod $p$ of the characteristic polynomial of a Frobenius element at $\ell$ is irreducible with nonzero trace.
	\end{enumerate}
	Then $\rho$ is surjective.
\end{lemma}

\begin{proof}
	This is just a reformulation of \cite[Corollary 2.2]{AK13}, where condition $(i)$ forces the existence of a transvection (cf. proof of Theorem \ref{tamemodp}).
\end{proof}

Condition $(i)$ is easy to force at some large prime $p_1>\max (2g+1,p)$ so it just remains to exhibit curves which satisfy the second condition for each of our exceptional cases in order to give an affirmative answer to the tame inverse Galois problem in these cases as well. In the table below, we give polynomials $f$ defining hyperelliptic curves, and a prime $\ell$ such that the image of $\Frob_{\ell}$ has the properties required for condition $(ii)$.

\begin{center}
\begin{tabular}{c|c|c}
	$(g,p)$ & $f(x)$ & $\ell$ \\ \hline
	$(3,7)$ & $x^7+x^3+3x^2+x+1$ & $3$ \\
	$(4,5)$ & $x^9+x^3+x^2+x+1$ & $3$ \\
	$(4,7)$ & $x^9+2x^3+2x^2+x+1$ & $3$ \\
	$(5,11)$ & $x^{11}+x^3+3x^2+x+1$ & $3$ \\
	$(7,5)$ & $x^{15}+3x^3+x^2+3x+1$ & $3$ \\
	$(7,11)$ & $x^{15}+4x^3+x^2+5x+1$ & $5$ \\
	$(7,13)$ & $x^{15}+2x^3+2x^2+2x+1$ & $3$ \\
	$(13,11)$ & $x^{27}+x^3+2x^2+2x+1$ & $5$ \\
	$(13,17)$ & $x^{27}+x^3+2x^2+x+1$ & $5$
\end{tabular}
\end{center}

Lastly we note that we are unable to do anything in the case $p=2$ since for a hyperelliptic curve, the image of the mod $2$ representation is always contained in a subgroup isomorphic to the symmetric group $S_{2g+2}$ and hence will never be surjective for $g \geqslant 3$.

\bibliographystyle{alpha}
\bibliography{mumbib}
\end{document}